\documentclass[11pt]{article}
\usepackage{amssymb, amsmath, amsthm, amsfonts,xcolor, enumerate}
\usepackage{tikz}
\usetikzlibrary{patterns}
\usetikzlibrary{shapes}
\usepackage{fullpage}
\usepackage{hyperref}
\newtheorem{theorem}{Theorem}[section]
\newtheorem{lemma}[theorem]{Lemma}

\theoremstyle{definition}

\newcommand{\eps}{\varepsilon}

\newcommand{\al}{\alpha}
\newcommand{\be}{\beta}

\newcommand{\De}{\Delta}

\newcommand{\cA}{\mathcal{A}}
\newcommand{\cB}{\mathcal{B}}

\newcommand{\cH}{\mathcal{H}}
\newcommand{\cS}{\mathcal{S}}
\newcommand{\cF}{\mathcal{F}}
\newcommand{\cG}{\mathcal{G}}

\newcommand{\cE}{\mathcal{E}}

\newcommand{\bekr}{\binom{n-1}{k-1}}
\newcommand{\bnkk}{\binom{n-k-1}{k-1}}

\renewcommand{\l}{\ensuremath{\ell}}

\newcommand{\card}[1]{\left| #1 \right|}

\newcommand{\disj}{\mathrm{dp}}
\newcommand{\Prb}{\mathbb{P}}
\newcommand{\Exp}{\mathbb{E}}

\title{Removal and stability for Erd\H{o}s--Ko--Rado}
\author{
Shagnik Das
\thanks{Freie Universit\"at Berlin, Institut f\"ur Mathematik, Arnimallee 3, 14195 Berlin, Germany.}
\thanks{{\tt shagnik@mi.fu-berlin.de}}
\and
Tuan Tran
\footnotemark[1]
\thanks{{\tt tuan@math.fu-berlin.de}. Research supported by DFG within the Research Training Grant ``Methods for Discrete Structures".}
}

\begin{document}
\maketitle

\begin{abstract}
A $k$-uniform family of subsets of $[n]$ is \emph{intersecting} if it does not contain a disjoint pair of sets.  The study of intersecting families is central to extremal set theory, dating back to the seminal Erd\H{o}s--Ko--Rado theorem of 1961 that bounds the size of the largest such families.  A recent trend has been to investigate the structure of set families with few disjoint pairs.  Friedgut and Regev proved a general removal lemma, showing that when $\gamma n \le k \le (\tfrac12 - \gamma)n$, a set family with few disjoint pairs can be made intersecting by removing few sets.

We provide a simple proof of a removal lemma for large families, showing that families of size close to $\ell \binom{n-1}{k-1}$ with relatively few disjoint pairs must be close to a union of $\ell$ stars.  Our lemma holds for a wide range of uniformities; in particular, when $\ell = 1$, the result holds for all $2 \le k < \frac{n}{2}$ and provides sharp quantitative estimates.

We use this removal lemma to answer a question of Bollob\'as, Narayanan and Raigorodskii regarding the independence number of random subgraphs of the Kneser graph $K(n,k)$.  The Erd\H{o}s--Ko--Rado theorem shows $\alpha(K(n,k)) = \binom{n-1}{k-1}$.  For some constant $c > 0$ and $k \le cn$, we determine the sharp threshold for when this equality holds for random subgraphs of $K(n,k)$, and provide strong bounds on the critical probability for $k \le \tfrac12 (n-3)$.
\end{abstract}

\section{Introduction} \label{sec:intro}

Extremal set theory, with its many connections and applications to other areas, has experienced remarkable growth in recent decades.  Inspired by one of the cornerstones of the field, the celebrated Erd\H{o}s--Ko--Rado theorem of 1961, a great deal of research concerns intersecting families.  We say a family of sets is \emph{intersecting} if it does not contain a pair of disjoint sets.  In this paper we derive a stability result for large families that are nearly intersecting, and apply it to obtain a sparse extension of the Erd\H{o}s--Ko--Rado theorem.  We begin with a brief survey of relevant results.

\subsection{Intersecting families and stability} \label{subsec:intro1}

We restrict our attention to $k$-uniform families of subsets of $[n]$.  The natural extremal question is to ask how large such a family can be if it is intersecting.  When $n < 2k$, there are no two disjoint sets, and hence $\binom{[n]}{k}$ itself is intersecting.  For $n \ge 2k$, a natural construction is to take all sets containing some fixed element $i \in [n]$.  This family, called the \emph{star with centre} $i$, contains $\binom{n-1}{k-1}$ sets, and Erd\H{o}s, Ko and Rado~\cite{ekr61} showed this is best possible.

Given the extremal result, great efforts have been made to better understand the general structure of large intersecting families.  Hilton and Milner~\cite{hm67} determined the size of the largest intersecting family that is not a subset of a star, before Frankl~\cite{fra87} extended this to determine the size of the largest intersecting family not containing too large a star.

In the years since these initial papers appeared, a series of stability results have been obtained.  Friedgut~\cite{fri08} and Dinur and Friedgut~\cite{df09} used spectral techniques to show, provided $k \le (\frac12 - \gamma)n$ for some $\gamma > 0$, any intersecting family of size close to $\binom{n-1}{k-1}$ is almost entirely contained in a star.  Keevash and Mubayi~\cite{km10} and Keevash~\cite{kee08} combined these methods with combinatorial arguments to provide similar results when $k$ is close to $\tfrac12 n$.

However, a recent trend in extremal set theory is to go beyond the Erd\H{o}s--Ko--Rado threshold and study set families that may not be intersecting, but contain few disjoint pairs.  Das, Gan and Sudakov~\cite{dgs14} studied the supersaturation problem, determining the minimum number of disjoint pairs appearing in sufficiently sparse $k$-uniform families.  Furthermore, a probabilistic variant of this supersaturation problem was introduced by Katona, Katona and Katona~\cite{kkk12}, and further studied by Russell~\cite{rus12}, Russell and Walters~\cite{rw13} and Das and Sudakov~\cite{ds14}.

Another direction that has been pursued is the transferral of the Erd\H{o}s--Ko--Rado theorem to the sparse random setting.  This study was initiated by Balogh, Bohman and Mubayi~\cite{bbm09}, who asked when the largest intersecting subfamily of a random $k$-uniform hypergraph is the largest star.  Progress on this problem has been made in subsequent papers by Gauy, H\`an and Oliveira~\cite{gho14}, Balogh, Das, Delcourt, Liu and Sharifzadeh~\cite{bddls14} and Hamm and Kahn~\cite{hk14a,hk14b}.  An alternative version of a sparse Erd\H{o}s--Ko--Rado theorem, which we shall discuss in greater detail in Section~\ref{subsec:intro3}, was introduced by Bollob\'as, Narayanan and Raigorodskii~\cite{bnr14}.

\subsection{Removal lemmas for disjoint pairs} \label{subsec:intro2}

As these new problems go beyond the Erd\H{o}s--Ko--Rado threshold, we require more robust forms of stability that apply not only to intersecting families, but also to families with few disjoint pairs.  This motivated the search for a \emph{removal lemma} that would show one can remove few sets from any family with a small number of disjoint pairs to obtain an intersecting family.  Such a result would be a set-theoretic analogue of the graph removal lemmas that have found a wide range of applications in extremal graph theory, details of which are in the survey of Conlon and Fox~\cite{cf13}.

Friedgut and Regev~\cite{fr14} proved the first such removal lemma, stated below.

\begin{theorem}[Friedgut--Regev] \label{thm:frremoval}
Let $\gamma > 0$, and let $k$ and $n$ be positive integers satisfying ${\gamma n \le k \le (\tfrac12 - \gamma)n}$.  Then for every $\eps > 0$ there is a $\delta > 0$ such that any family $\cF \subset \binom{[n]}{k}$ with at most $\delta \card{\cF} \binom{n-k}{k}$ disjoint pairs can be made intersecting by removing at most $\eps \binom{n-1}{k-1}$ sets from $\cF$.
\end{theorem}

This is a very general result that holds regardless of the size or structure of the nearest intersecting family.  However, for extremal applications, one is typically interested in the case when $\card{\cF} \approx \binom{n-1}{k-1}$.  For example, Gauy, H\`an and Oliveira required such a lemma in~\cite{gho14}, coupling Theorem~\ref{thm:frremoval} with known stability results to show that a family of size close to $\binom{n-1}{k-1}$ with few disjoint pairs must be close in structure to a star.  They further asked if such a result also holds for $k=o(n)$.  Our main theorem shows this is indeed the case.  Theorem~\ref{thm:removal} provides a removal lemma that holds whenever $\cF$ has size close to a union of $\ell$ full stars and has relatively few disjoint pairs.  Moreover, when $\ell = 1$, this holds for the full range of $2 \le k < \frac{n}{2}$.

\begin{theorem} \label{thm:removal}
There is an absolute constant $C > 1$ such that if $n,k$ and $\l$ are positive integers satisfying $n>2k\l^2$, and $\cF \subset \binom{[n]}{k}$ is a family of size $\card{\cF} = \left(\l - \al \right) \bekr$ with at most $\left(\binom{\l}{2}+\be\right) \bekr \bnkk$ disjoint pairs, where $\max \left(2\l\card{\al}, \card{\be} \right) \le \frac{n-2k}{(20C)^2n}$, then there is a family $\cS$ that is the union of $\l$ stars satisfying $\card{\cF \De \cS} \le C((2\l-1)\al + 2\be)\frac{n}{n-2k}\bekr$.
\end{theorem}

Another feature of Theorem~\ref{thm:removal} is that, despite its simple proof, it provides quantitative control that is often sharp up to the constant.  The distance from $\cF$ to a union of $\ell$ stars is measured in terms of its size (parametrised by $\alpha$), the number of disjoint pairs (parametrised by $\beta$), and how close $k$ is to $\tfrac12 n$.  When $\ell = 1$, taking $\beta = 0$ gives a stability result for intersecting families, and the bounds sharpen those given by Keevash and Mubayi~\cite{km10} and Keevash~\cite{kee08}.

For positive $\beta$, the bounds remain sharp up to the constant.  If $k$ is bounded away from $\frac{n}{2}$, then one may take a star and add $\alpha \binom{n-1}{k-1}$ sets from another star to obtain a family of size $(1 + \alpha) \binom{n-1}{k-1}$ with $\alpha \binom{n-1}{k-1} \binom{n-k-1}{k-1}$ disjoint pairs that is $\alpha \binom{n-1}{k-1}$-far from a star.  On the other hand, if $t = n - 2k = o(n)$, consider the anti-star $\binom{[n-1]}{k}$.  This has size $\left(1 + \frac{t}{k} \right) \binom{n-1}{k-1}$, contains approximately $\frac{t}{n} \binom{n-1}{k-1} \binom{n-k-1}{k-1}$ disjoint pairs, and yet is approximately $\binom{n-1}{k-1}$-far from a star.

When $\ell \ge 2$, $\cF$ is much larger than the Erd\H{o}s--Ko--Rado bound, and hence we would expect $\cF$ to contain many disjoint pairs.  Das, Gan and Sudakov~\cite{dgs14} have shown that, provided $n$ is sufficiently large, a union of $\ell$ stars, which has approximately $\binom{\ell}{2} \binom{n-1}{k-1} \binom{n-k-1}{k-1}$ disjoint pairs, minimises the number of disjoint pairs in set families of this size.  Theorem~\ref{thm:removal} provides stability for this supersaturation result, showing that families of comparable size with a similar number of disjoint pairs must be close in structure to a union of $\ell$ stars.

Finally, while we require $n > 2k$ when $\ell = 1$, we can do a bit better when $\ell$ is large: as $\ell$ tends to infinity, the bound on $n$ can be lowered to $n > \left( \tfrac12 + o(1) \right) k \ell^2$.

\subsection{Erd\H{o}s--Ko--Rado for sparse Kneser subgraphs} \label{subsec:intro3}

To demonstrate the usefulness of Theorem~\ref{thm:removal}, we shall apply it to a problem of Bollob\'as, Narayanan and Raigorodskii~\cite{bnr14} regarding an extension of the Erd\H{o}s--Ko--Rado theorem to the sparse random setting.  To define the problem at hand, we first need to introduce the Kneser graph and its connection to the Erd\H{o}s--Ko--Rado theorem.

Given integers $1 \le k \le \tfrac12 n$, the Kneser graph $K(n,k)$ is defined on the vertex set $V = \binom{[n]}{k}$, with two $k$-sets $F, G \in \binom{[n]}{k}$ adjacent in $K(n,k)$ if and only if $F \cap G = \emptyset$.  Since edges of $K(n,k)$ denote disjoint pairs in $\binom{[n]}{k}$, it follows that independent sets of $K(n,k)$ correspond directly to intersecting families in $\binom{[n]}{k}$.  Thus the Erd\H{o}s--Ko--Rado theorem, viewed from the perspective of the Kneser graph, shows $\alpha(K(n,k)) = \binom{n-1}{k-1}$ when $n \ge 2k$.

Bollob\'as, Narayanan and Raigorodskii~\cite{bnr14} transferred the Erd\H{o}s--Ko--Rado theorem to the random setting by considering not the entire Kneser graph $K(n,k)$, but rather random subgraphs of it.  Given some probability $0 \le p \le 1$, let $K_p(n,k)$ denote the subgraph of $K(n,k)$ where every edge is retained independently with probability $p$.  As $K_p(n,k) \subseteq K(n,k)$, we clearly have $\alpha(K_p(n,k)) \ge \alpha(K(n,k)) = \binom{n-1}{k-1}$.  They then asked for which $p$ we have equality.

In their paper, they showed the Erd\H{o}s--Ko--Rado theorem is surprisingly robust when $k$ is not too large with respect to $n$.  In other words, we almost surely have $\alpha(K_p(n,k)) = \binom{n-1}{k-1}$ even for very small probabilities $p$ (and thus very sparse subgraphs of $K(n,k)$).  Furthermore, they exhibited a sharp threshold for when this sparse Erd\H{o}s--Ko--Rado theorem holds.

\begin{theorem}[Bollob\'as--Narayanan--Raigorodskii] \label{thm:bnr14}
Fix $\eps > 0$ and suppose $2 \le k = o(n^{1/3})$.  Let
\[ p_0 = \frac{(k+1) \log n - k \log k}{\binom{n-1}{k-1}}. \]
Then, as $n \rightarrow \infty$,
\[ \Prb \left( \alpha(K_p(n,k)) = \binom{n-1}{k-1} \right) \rightarrow \left\{ \begin{array}{ll}
	0 & \textrm{if } p \le (1 - \eps) p_0  \\
	1 & \textrm{if } p \ge (1 + \eps) p_0 
\end{array} \right. . \]
Moreover, for $p \ge (1 + \eps)p_0$, with high probability the largest independent sets are stars.
\end{theorem}

While observing that we may take $\eps = O(k^{-1})$, they conjectured that the result should continue to hold provided $k = o(n)$.  Partial progress was made by Balogh, Bollob\'as and Narayanan~\cite{bbn14}, who showed that for every $\gamma > 0$ there is some constant $c(\gamma) > 0$ such that if $k \le (\tfrac12 - \gamma) n$ and $p \ge \binom{n-1}{k-1}^{-c(\gamma)}$, then $\alpha(K_p(n,k)) = \binom{n-1}{k-1}$ with high probability.

By applying Theorem~\ref{thm:removal}, we obtain sharper results for large $k$, as given in the theorem below.  For these larger values of $k$, it is convenient to present the critical probability in a different form to that of Theorem~\ref{thm:bnr14}; note that $p_c$ below is asymptotically equal to $p_0$ above when $k = o(n^{1/2})$.

\begin{theorem} \label{thm:sparseEKR}
There is an absolute constant $C > 0$ such that the following holds.  Let $k$ and $n$ be integers with $\omega(1) = k \le \tfrac12 (n-3)$, let $\eps = \omega(k^{-1})$, and set
\[ p_c = \frac{\log \left( n \binom{n-1}{k} \right) }{\binom{n-k-1}{k-1}}. \]
Then, as $n \rightarrow \infty$, $\Prb \left( \alpha(K_p(n,k)) = \binom{n-1}{k-1} \right) \rightarrow 0$ if ${p \le (1 - \eps) p_c}$.

For $k \le \frac{n}{6C}$, if ${p \ge (1 + \eps) p_c}$, with high probability $\alpha(K_p(n,k)) = \binom{n-1}{k-1}$ and the stars are the only maximum independent sets.  For $k \le \tfrac12 (n-3)$, the same conclusion holds for $p \ge \frac{2Cn}{n-2k} p_c$.
\end{theorem}

Theorem~\ref{thm:sparseEKR} exhibits a sharp threshold for $k \le \frac{n}{6C}$, thus extending Theorem~\ref{thm:bnr14} to $k$ as large as linear in $n$.  Furthermore, when $k \le (\tfrac12 - \gamma)n$, as considered in~\cite{bbn14}, $\frac{n}{n-2k} \le (2\gamma)^{-1}$, and so Theorem~\ref{thm:sparseEKR} determines the critical probability up to a constant factor.  Finally, when $k$ is close to $\tfrac12 n$, we find that the sparse version of the Erd\H{o}s--Ko--Rado theorem still holds for very small edge probabilities; when $k = \tfrac12 (n-3)$, we almost surely have $\alpha(K_p(n,k)) = \binom{n-1}{k-1}$ even for $p = \Omega(n^{-1})$.

\subsection{Notation and organisation} \label{subsec:intro4}

We use standard set-theoretic and asymptotic notation.  We write $[n]$ for $\{1, 2, \hdots, n\}$, while $\binom{X}{k}$ is the family of all $k$-element subsets of a set $X$.  A non-standard piece of notation is $\disj(\cF)$, which represents the number of disjoint pairs in a set family $\cF$.  Given two functions $f$ and $g$ of some underlying parameter $n$, if $\lim_{n \rightarrow \infty} f(n) / g(n) = 0$, we write $f = o(g)$ and $g = \omega(f)$.  Finally, all logarithms are to the base $e$.

In Section~\ref{sec:removal} we prove our removal lemma, Theorem~\ref{thm:removal}.  We apply this result to the sparse Erd\H{o}s--Ko--Rado problem in Section~\ref{sec:kneser}, where we prove Theorem~\ref{thm:sparseEKR}.  The final section contains some concluding remarks and open problems.

\section{The removal lemma} \label{sec:removal}

In this section we prove our version of the removal lemma, Theorem~\ref{thm:removal}.  Our proof combines the work of Lov\'asz~\cite{lov79} on the spectrum of the Kneser graph with an analytic result of Filmus~\cite{fil14} regarding approximations of Boolean functions on $\binom{[n]}{k}$.  Before beginning with the proof, we shall introduce the necessary terminology.

Given a family of sets $\cF \subset \binom{[n]}{k}$, the \emph{characteristic function} $f: \binom{[n]}{k} \rightarrow \{0,1\}$ is a Boolean function indicating membership of the family, with $f(F) = 1$ if and only if $F \in \cF$.  We may embed $\binom{[n]}{k} \subset \{0,1\}^n$ into the $n$-dimensional hypercube, and thus think of $f$ as being defined on the $k$-uniform \emph{slice} of the cube $\{ (x_1, \hdots, x_n) \in \{0,1\}^n : \sum_i x_i = k \}$.  A function $f$ is \emph{affine} if $f(x_1, x_2, \hdots, x_n) = a_0 + \sum_{i \in [n]} a_i x_i$ for some constants $a_i$, $0 \le i \le n$.  We will equip this space of functions with the $L_2$-norm with respect to the uniform measure on $\binom{[n]}{k}$, defining
\[ \|f - g\|^2 = \Exp \left[ \card{f-g}^2 \right] = \frac{1}{\binom{n}{k}} \sum_{F \in \binom{[n]}{k}} \card{f(F) - g(F)}^2, \]
and say $f$ and $g$ are \emph{$\eps$-close} if $\|f - g\|^2 \le \eps$.  Finally, to avail of the spectral results, which are traditionally phrased in terms of matrices and vectors, we will abuse notation and identify a function $f: \binom{[n]}{k} \rightarrow \mathbb{R}$ with the vector $f = (f(F))_{F \in \binom{[n]}{k}}$ in $\mathbb{R}^{\binom{[n]}{k}}$.  Note that the $L_2$-norm above arises from the standard inner product on $\mathbb{R}^{\binom{[n]}{k}}$.

The first step of our proof is the following lemma, which transfers the problem into the analytic framework set up above.  The lemma shows that if a set family $\cF$ is as in the statement of Theorem~\ref{thm:removal}, then its characteristic function can be approximated well by an affine function.

\begin{lemma} \label{lem:spectral}
Let $n$, $k$ and $\l$ be positive integers satisfying $n>2k$, and let $\cF \subset \binom{[n]}{k}$ be a family of size $\card{\cF} = (\l - \al) \bekr$ with at most $\left(\binom{\l}{2}+\be\right)\bekr \bnkk$ disjoint pairs.  If ${f : \binom{[n]}{k} \rightarrow \{0,1\}}$ is the characteristic function of $\cF$, then $\| f - g \|^2 \le \left((2\l-1) \al + 2\be\right) \frac{k}{n-2k}$ for some affine function $g: \binom{[n]}{k} \rightarrow \mathbb{R}$.
\end{lemma}

To prove Lemma~\ref{lem:spectral}, we require some information on the spectrum of the Kneser graph.  Let $A$ denote the adjacency matrix of $K(n,k)$.  In his celebrated paper on the Shannon capacity of graphs, Lov\'asz~\cite{lov79} showed the eigenvalues of $A$ are $\lambda_i = (-1)^i \binom{n-k-i}{k-i}$ for $0 \le i \le k$.  Thus the largest eigenvalue is the degree of the vertices in the regular graph $K(n,k)$, $\lambda_0 = \binom{n-k}{k}$, while the smallest eigenvalue is $\lambda_1 = - \binom{n-k-1}{k-1}$.  The second smallest eigenvalue is $\lambda_3 = - \binom{n-k-3}{k-3}$.  Furthermore, the $\lambda_0$-eigenspace is one-dimensional, spanned by the constant function.  The $(n-1)$-dimensional $\lambda_1$-eigenspace is spanned by the functions $x_i - \frac{k}{n}$ (note that these functions are linearly dependent, as $\sum_i x_i \equiv k$).  Hence the span of the $\lambda_0$- and $\lambda_1$-eigenspaces is precisely the space of affine functions.  As $A$ is a real symmetric matrix, its eigenspaces are orthogonal.  Armed with these preliminaries, we can prove the lemma.

\begin{proof}[Proof of Lemma~\ref{lem:spectral}]
Given the characteristic vector $f$ of $\cF$, write $f = f_0 + f_1 + f_2$, where $f_0$ and $f_1$ are the projections of $f$ to the $\lambda_0$- and $\lambda_1$-eigenspaces respectively, and $f_2 = f - f_0 - f_1$.  By the orthogonality of the eigenspaces, we have $\| f \|^2 = \| f_0 \|^2 + \| f_1 \|^2 + \| f_2 \|^2$.  As $f$ is a Boolean function, $\| f \|^2 = \Exp[ f^2 ] = \Exp[ f] = \card{\cF} / \binom{n}{k} = (\l-\alpha) \frac{k}{n}$.  Thus, solving for $\|f_1\|^2$, we find $\| f_1 \|^2 = (\l - \alpha) \frac{k}{n} - \| f_0 \|^2 - \| f_2 \|^2$.  Furthermore, since the $\lambda_0$-eigenspace is spanned by the constant function, $f_0 \equiv \Exp[f] = (\l - \alpha) \frac{k}{n}$, and so $\|f_0\|^2 = \Exp \left[ f_0^2 \right] = (\l-\alpha)^2 \frac{k^2}{n^2}$.

As $A$ is the adjacency matrix of the Kneser graph $K(n,k)$, and $f$ is the characteristic function of the set family $\cF$, it follows that $f^T A f = 2 \disj(\cF)$.  Using our bound on the number of disjoint pairs in $\cF$,
\begin{align*}
	\left( \l^2 - \l + 2 \beta \right) \binom{n-1}{k-1} \binom{n-k-1}{k-1} &\ge 2 \disj(\cF) = f^T A f = f_0^T A f_0 + f_1^T A f_1 + f_2^T A f_2 \\
	&\ge \lambda_0 f_0^T f_0 + \lambda_1 f_1^T f_1 + \lambda_3 f_2^T f_2.
\end{align*}
We divide through by $\binom{n}{k}$ to normalise, obtaining
\[ \frac{\left( \l^2 - \l + 2 \beta \right) k}{n} \binom{n-k-1}{k-1} \ge \binom{n-k}{k} \| f_0 \|^2 - \binom{n-k-1}{k-1} \| f_1 \|^2 - \binom{n-k-3}{k-3} \| f_2 \|^2. \]
Dividing by $\binom{n-k-1}{k-1}$, substituting our expressions for $\| f_0 \|^2$ and $\| f_1 \|^2$, and simplifying gives
\begin{align*}
	\frac{2 \beta k}{n} &\ge \left[ 1 - \frac{(k-1)(k-2)}{(n-k-1)(n-k-2)} \right] \| f_2 \|^2 - \frac{\left( 2\l - 1\right) \alpha k}{n} + \frac{\al^2 k}{n}\\
	&= \frac{(n-2k)(n-3) }{(n-k-1)(n-k-2)}\| f_2 \|^2 - \frac{\left( 2 \l - 1 \right) \al k}{n} + \frac{\al^2 k}{n} \ge \frac{n-2k}{n} \| f_2 \|^2 - \frac{ \left( 2 \l - 1 \right)\alpha k}{n}.
\end{align*}
Rearranging, we deduce $\| f_2 \|^2 \le ((2 \l - 1) \alpha + 2 \beta)\frac{k}{n-2k}$. Recalling that $f_0 + f_1$ is spanned by the $\lambda_0$- and $\lambda_1$-eigenspaces, and hence affine, setting $g = f_0 + f_1$ gives the desired result.
\end{proof}

Lemma~\ref{lem:spectral} shows the characteristic function of $\cF$ must be close to an affine function, from which we shall deduce that $\cF$ itself is close to a union of stars.  Note that the characteristic function $g$ of the union of stars with centres $i \in S$ is simply $g(x_1, \hdots, x_n) = \max_{i \in S} x_i$, and is thus determined only by the coordinates in $S$.  The Friedgut--Kalai--Naor theorem~\cite{fkn02} states that if a Boolean function $f: \{0,1\}^n \rightarrow \{0,1\}$ on the entire hypercube is close to an affine function, then it is close to a function determined by at most one coordinate.  We shall make use of an analogous result for the $k$-uniform slices of the cube, due to Filmus~\cite{fil14}.

\begin{theorem}[Filmus] \label{thm:filmus}
For some constant $C > 1$ the following holds.  Suppose $2 \le k \le \tfrac12 n$ and $\eps < \frac{k}{128n}$.  For every Boolean function $f: \binom{[n]}{k} \rightarrow \{0,1\}$ that is $\eps$-close to an affine function, there is some set $S \subset [n]$ of size $\card{S} \le \max \left( 1, \frac{Cn \sqrt{\eps}}{k} \right)$ such that either $f$ or $1-f$ is $(C \eps)$-close to $\max_{i \in S} x_i$.
\end{theorem}

We now have all the necessary ingredients to prove the removal lemma.

\begin{proof}[Proof of Theorem~\ref{thm:removal}]
Set $\eps = ((2\l - 1)\alpha + 2 \beta)\frac{k}{n-2k}$, and take $C$ as in Theorem~\ref{thm:filmus}.  By our bounds on $\alpha$ and $\beta$, $\eps < \frac{k}{128C^2 n}$.  If $\cF$ is as in the statement of the theorem, then by Lemma~\ref{lem:spectral} its characteristic function $f$ is $\eps$-close to an affine function.  By Theorem~\ref{thm:filmus}, there is some $S \subset [n]$ such that $f$ or $1-f$ is $(C\eps)$-close to $\max_{i \in S} x_i$.  Without loss of generality, we may assume $S = [s]$, where $s \le \max \left(1 , \frac{Cn \sqrt{\eps}}{k} \right)$.  Let $g_s = \max_{i \in [s]} x_i$, and let $\cG_s = \binom{[n]}{k} \setminus \binom{[n] \setminus [s]}{k}$ be the family corresponding to this characteristic function.

Note that $\| f - g_s \|^2 = \card{\cF \Delta \cG_s} / \binom{n}{k}$, since for any set $F \in \binom{[n]}{k}$ we have
\[ \card{ f(F) - g_s(F) } = \left\{ \begin{array}{ll}
	1 & \textrm{if } F \in \cF \Delta \cG_s \\
	0 & \textrm{otherwise}
	\end{array} \right. . \]
Hence we must have $\card{\cF \Delta \cH} \le C \eps \binom{n}{k}$ for $\cH = \cG_s$ or $\cH = \overline{\cG_s}$, depending on whether it is $f$ or $1-f$ that is $(C \eps)$-close to $g_s$.  There are six possibilities to consider:
\begin{equation*}
\begin{tabular}{rlrlrl}
	(i)&$\cH = \cG_s, s \le \l - 1$ \hspace{1.0cm} &(ii)&$\cH = \cG_s, s \ge \l + 1$ \hspace{1.0cm} &(iii)&$\cH = \overline{\cG_0}$ \\
	(iv)&$\cH = \overline{\cG_s}, s \ge 2$ &(v)&$\cH = \overline{\cG_1}$ &(vi)&$\cH = \cG_{\l}$
\end{tabular}
\end{equation*}
Since $\cG_{\l}$ is the union of $\l$ stars, we wish to show that (vi) must hold.  We first consider the sizes of $\cF$ and $\cH$ to eliminate cases (i)-(iv).  Recall that $\card{\cF} = (\l - \alpha) \binom{n-1}{k-1}$, and, by our bound on $\alpha$, $\l - \alpha \in ( \l - \frac18, \l + \frac18)$.  Since $\card{ \card{\cF} - \card{\cH} } \le \card{ \cF \Delta \cH} \le C \eps \binom{n}{k} < \frac18 \binom{n-1}{k-1}$, we must have $\left( \l - \frac14 \right) \binom{n-1}{k-1} \le \card{\cH} \le \left( \l + \frac14 \right) \binom{n-1}{k-1}$.

We have $\card{\cG_s} \le s \binom{n-1}{k-1}$, which is too small if $s \le \l - 1$.  On the other hand, observe that $\cG_s$, the union of $s$ stars, grows with $s$.  Thus, when $s \ge \l + 1$,
\[ \card{\cG_s} \ge \card{\cG_{\l + 1}} \ge (\l + 1) \binom{n-1}{k-1} - \binom{\l + 1}{2} \binom{n-2}{k-2} \ge \left(\l + 1 - \frac{\l^2k}{n} \right) \binom{n-1}{k-1} \ge \left( \l + \frac12 \right) \binom{n-1}{k-1}, \]
which is too large.  This rules out cases (i) and (ii).  We also have $\card{\overline{\cG_0}} = \binom{n}{k} = \frac{n}{k} \binom{n-1}{k-1} \ge 2 \l^2 \binom{n-1}{k-1}$, which is again too large, ruling out case (iii) as well.

To handle case (iv), we show that $\overline{\cG_s}$ is too large when $s \ge 2$.  Since $\card{\overline{\cG_s}} = \binom{n-s}{k}$ is decreasing in $s$, it suffices to take $s = \frac{C n \sqrt{\eps}}{k}$.  We indeed have too many sets, as
\[ \card{\overline{\cG_s}} = \binom{n-s}{k} \ge \left( 1 - \frac{sk}{n} \right) \binom{n}{k} = \left( 1 - C \sqrt{\eps} \right) \frac{n}{k} \binom{n-1}{k-1} > \frac{3 \l^2}{2} \binom{n-1}{k-1} \ge \left( \l + \frac12 \right) \binom{n-1}{k-1}. \]

The above argument does not immediately rule out case (v), since if $s = \max \left( 1, \frac{C n \sqrt{\eps}}{k} \right) = 1$, we may not assume $s = \frac{C n \sqrt{\eps}}{k}$.  However, the family $\overline{\cG_1}$ is still too large when $\l \ge 2$, as
\[ \card{\overline{\cG_1}} = \binom{n-1}{k} = \frac{n-k}{k} \binom{n-1}{k-1} \ge \left( 2\l^2 - 1 \right) \binom{n-1}{k-1} > \left( \l + \frac12 \right) \binom{n-1}{k-1}. \]

To rule out case (v) when $\l = 1$, we consider the number of disjoint pairs in $\cF$.  Note that each of the $\binom{n-1}{k}$ sets in $\overline{\cG_1}$ is disjoint from $\binom{n-k-1}{k}$ other sets in $\overline{\cG_1}$, and hence $\disj(\overline{\cG_1}) = \frac12 \binom{n-1}{k} \binom{n-k-1}{k}$.  Moreover, removing $t$ sets from $\overline{\cG_1}$ can account for at most $t \binom{n-k-1}{k}$ disjoint pairs.  If $\cF$ were close to $\overline{\cG_1}$, then $\card{\overline{\cG_1} \setminus \cF} \le C\eps \binom{n}{k}$, and so
\[ \disj(\cF) \ge \disj(\cF \cap \overline{\cG_1}) \ge \left( \tfrac12 - \tfrac{C \eps n}{n-k} \right) \binom{n-1}{k} \binom{n-k-1}{k} > \left( \tfrac12 - 2C \eps \right) \binom{n-1}{k} \binom{n-k-1}{k}. \]
On the other hand, we assumed $\cF$ has at most $\beta \binom{n-1}{k-1} \binom{n-k-1}{k-1}$ disjoint pairs, so we must have $\beta \ge \left( \frac12 - 2C \eps \right) \frac{(n-k)(n-2k)}{k^2} > \frac{n-2k}{2n}$, contradicting our bound on $\beta$.

Thus we are only left with case (vi), where $\cH$ is the union of $\l$ stars $\cG_{\l}$, and, as required, we have ${\card{ \cF \Delta \cG_{\l}} \le C \eps \binom{n}{k} = C((2 \l - 1) \alpha + 2 \beta)\frac{n}{n-2k} \binom{n-1}{k-1}}$.
\end{proof}

\section{Independence number of random Kneser subgraphs} \label{sec:kneser}

In this section we prove Theorem~\ref{thm:sparseEKR}, establishing an analogue of the Erd\H{o}s--Ko--Rado theorem for sparse random subgraphs of the Kneser graph.  We will show that below the threshold, there is with high probability an independent \emph{superstar}: some star $\cS$ and a set $F \notin \cS$ such that $\cS \cup \{ F \}$ is independent in $K_p(n,k)$.  The upper bound on the critical probability essentially follows from a union bound over all potential independent sets, where we will be able to take advantage of the fine control afforded to us by Theorem~\ref{thm:removal} to obtain sharp results when $k$ is large.

\begin{proof}[Proof of Theorem~\ref{thm:sparseEKR}]
First we establish the lower bound on the critical probability.  Suppose ${\eps = \omega(k^{-1})}$ and $p \le (1 - \eps) p_c$.  We wish to show that with high probability, stars can be extended to independent superstars in $K_p(n,k)$.  Let $\cS$ be the star with centre $1$, and for every $1 \notin F \in \binom{[n]}{k}$ let $\cE_F$ be the event that $\cS \cup \{ F \}$ is independent in $K_p(n,k)$.

Note that $F$ is disjoint from $\binom{n-k-1}{k-1}$ sets in $\cS$, and for $\cE_F$ to hold none of these edges can appear in $K_p(n,k)$.  Thus $\Prb(\cE_F) = (1-p)^{\binom{n-k-1}{k-1}}$.  Moreover, the events $\{ \cE_F : 1 \notin F \}$ depend on mutually disjoint sets of edges of $K(n,k)$, and are thus independent.  Hence we can bound the probability that the stars are the largest independent sets of $K_p(n,k)$ by
\[ \Prb\left( \alpha(K_p(n,k)) = \binom{n-1}{k-1} \right) \le \Prb\left( \cap_F \overline{\cE_F} \right) =  \left(1 - (1-p)^{\binom{n-k-1}{k-1}} \right)^{\binom{n-1}{k}}. \]

This bound is increasing in $p$, so it suffices to take $p = (1 - \eps) p_c = \frac{(1 - \eps) \log \left( n \binom{n-1}{k} \right)}{\binom{n-k-1}{k-1}}$.  As $n \ge 2k + 2$, $p = O(n^{-1}) = o(\eps)$, and hence $(1-p)^{\binom{n-k-1}{k-1}} \ge \mathrm{exp} \left( - p(1 + p) \binom{n-k-1}{k-1} \right) \ge \left(n \binom{n-1}{k} \right)^{-(1 - \eps/2)}$.  Thus
\[ \left(1 - (1-p)^{\binom{n-k-1}{k-1}} \right)^{\binom{n-1}{k}} \le \mathrm{exp} \left( - \binom{n-1}{k} (1-p)^{\binom{n-k-1}{k-1}} \right) \le \mathrm{exp} \left( - n^{-1} \binom{n-1}{k} ^{\eps/2} \right) = o(1), \]
since $\eps = \omega \left(k^{-1} \right)$.  Hence for $p \le (1 - \eps) p_c$ we have $\alpha(K_p(n,k)) > \binom{n-1}{k-1}$ with high probability.

We now seek an upper bound on the critical probability.  By monotonicity, it suffices to consider $p$ as small as possible.  To begin, we will prove the coarse threshold.  Let $C$ be the (absolute) constant from Theorem~\ref{thm:removal}, and take $p = \zeta p_c$, where $\zeta = \frac{2Cn}{n-2k}$.  For such $p$, we wish to show the only maximum independent sets of $K_p(n,k)$ are the stars.  To this end, we define the following random variables:
\begin{align*}
	X &= \card{ \left\{ \textrm{independent superstars } \cF: \cF = \cS \cup \{F \} \textrm{ for some star } \cS, F \notin \cS \right\} } \textrm{ and} \\
	Y_i &= \card{ \left\{ \textrm{independent }\cF: \card{\cF} = \binom{n-1}{k-1}, \min_{\cS \textrm{ a star}} \card{ \cS \setminus \cF } = i \right\} }, \; 1 \le i \le \binom{n-1}{k-1}.
\end{align*}

$X$ counts the number of independent superstars in $K_p(n,k)$.  If $X = 0$, the stars are all maximal independent sets.  If we further have $Y_i = 0$ for all $1 \le i \le \binom{n-1}{k-1}$, then there are no independent sets of size $\binom{n-1}{k-1}$ that are not stars, and thus the stars are the only maximum independent sets in $K_p(n,k)$.  Hence our task is to show $X + \sum_i Y_i = 0$ with high probability.  By the union bound, it suffices to show $\Prb(X > 0) + \sum_i \Prb(Y_i > 0) = o(1)$.

We begin by estimating $\Prb(X > 0)$, which we can bound by $\Exp[X]$.  There are $n$ choices for the star $\cS$, $\binom{n-1}{k}$ choices for the set $F \notin \cS$, and, for each such configuration, $\binom{n-k-1}{k-1}$ edges that should not appear in $K_p(n,k)$, which occurs with probability $(1 - p)^{\binom{n-k-1}{k-1}} \le \mathrm{exp} \left( - \zeta p_c \binom{n-k-1}{k-1} \right)$. Thus
\begin{equation} \label{eqn:Xbound}
\Exp[X] \le n \binom{n-1}{k} \mathrm{exp} \left( - \zeta p_c \binom{n-k-1}{k-1} \right) = \left( n \binom{n-1}{k} \right)^{1 - \zeta} = o(1),
\end{equation}
even for $\zeta$ as small as $1 + \omega \left( k^{-1} \right)$.

To analyse $\Prb(Y_i > 0)$, we shall distinguish between two different cases: families that are close to a star, and families far from a star.  For the first case, we assume $1 < i \le t_1 = \frac{1}{400C} \binom{n-1}{k-1}$.  The families $\cF$ counted by $Y_i$ have size $\binom{n-1}{k-1}$ and $\card{\cF \Delta \cS} = 2 \card{\cS \setminus \cF} \ge 2i$ for every star $\cS$.  By applying Theorem~\ref{thm:removal} with $\alpha = 0$ and $\beta = \frac{i(n-2k)}{Cn \binom{n-1}{k-1}}$, it follows that $\disj(\cF) \ge \frac{i(n-2k)}{Cn} \binom{n-k-1}{k-1}$.  For $\cF$ to be independent in $K_p(n,k)$, none of these edges can appear, which occurs with probability $(1 - p)^{\disj(\cF)} \le \left( n \binom{n-1}{k} \right)^{ - \zeta i (n-2k)/ (Cn)}$.

We now take a union bound over all possible choices of $\cF$.  We know there is some star $\cS$ such that $\card{\cS \setminus \cF} = i$.  There are $n$ choices for the star $\cS$, $\binom{\binom{n-1}{k-1}}{i}$ choices for the $i$ sets in $\cS \setminus \cF$, and $\binom{\binom{n-1}{k}}{i}$ choices for the $i$ sets in $\cF \setminus \cS$.  Hence there are at most $n \binom{ \binom{n-1}{k-1} }{i} \binom{ \binom{n-1}{k} }{i} \le n \left( \frac{k e^2 \binom{n-1}{k}^2}{(n-k) i^2} \right)^i$ families $\cF$ that can be counted by $Y_i$.  Thus we have
\begin{equation} \label{eqn:medYbound}
\sum_{i = 1}^{t_1} \Prb(Y_i > 0) \le \sum_{i=1}^{t_1} n \left( \frac{ke^2 \binom{n-1}{k}^2 }{(n-k) i^2} \left(n \binom{n-1}{k} \right)^{- \zeta (n-2k)/(Cn) } \right)^i \le \sum_{i=1}^{t_1} \frac{ne^{2i}}{(ni)^{2i}} = o(1).
\end{equation}
where the second inequality follows from our choice of $\zeta = \frac{2Cn}{n-2k}$.

Finally, we bound $\Prb(Y_i > 0)$ when $i > t_1$.  Applying Theorem~\ref{thm:removal} with $\alpha = 0$ and $\beta = \frac{n-2k}{(20C)^2 n}$, any family $\cF$ counted by $\sum_{i > t_1} Y_i$ must have $\disj(\cF) \ge \frac{n-2k}{(20C)^2 n} \binom{n-1}{k-1} \binom{n-k-1}{k-1}$.  Hence the probability of such an $\cF$ being independent in $K_p(n,k)$ is $(1 - p)^{\disj(\cF)} \le \left(n \binom{n-1}{k} \right)^{- \frac{\zeta (n-2k)}{(20C)^2 n} \binom{n-1}{k-1}}$.  Recalling our choice of $\zeta = \frac{2Cn}{n-2k}$, we apply a trivial union bound, summing over all $\binom{\binom{n}{k}}{\binom{n-1}{k-1}} \le \left( \frac{ne}{k} \right)^{\binom{n-1}{k-1}}$  families of size $\binom{n-1}{k-1}$ to find, when $k \ge 600C$,
\begin{equation} \label{eqn:largeYbound}
\sum_{i = t_1+1}^{\binom{n-1}{k-1}} \Prb(Y_i > 0) \le \left( \frac{ne}{k} \left( n \binom{n-1}{k} \right)^{- 1/ (200C)} \right)^{\binom{n-1}{k-1}} \le \left( \frac{ne}{k} \left( \frac{k}{n} \right)^{k/(200C)} \right)^{\binom{n-1}{k-1}} = o(1).
\end{equation}
Combining~\eqref{eqn:Xbound},~\eqref{eqn:medYbound}, and~\eqref{eqn:largeYbound}, we find that when $p \ge \zeta p_c$, $\Prb(X > 0) + \sum_i \Prb(Y_i > 0) = o(1)$, and so for such $p$, the maximum independent sets in $K_p(n,k)$ are precisely the stars.

We now prove the sharp threshold result, for which we must show that the same conclusion holds when $k \le \frac{n}{6C}$ and $\zeta = 1 + \eps$, for some small $\eps = \omega(k^{-1})$.  As previously stated, the bound from~\eqref{eqn:Xbound} holds with this smaller value of $\zeta$.  However, bounding $\sum_i \Prb(Y_i > 0)$ requires more careful analysis.  We now split the sum into three parts: $1 \le i \le t_0 = \frac{\eps}{2} \binom{n-k-1}{k-1}$, $t_0 + 1 \le i \le t_1 = \frac{1}{400C} \binom{n-1}{k-1}$, and $t_1 + 1 \le i \le \binom{n-1}{k-1}$.

For the latter two parts, we modify slightly our analysis of the above bounds.  When a family is only moderately close to a star, with $t_0 + 1 \le i \le t_1$, we again use the bound in~\eqref{eqn:medYbound}.  We begin by observing that $i > t_0 = \frac{\eps}{2} \binom{n-k-1}{k-1} = \frac{\eps k}{2 (n-k)} \binom{n-k}{k}$, and so
\[ \frac{k e^2 \binom{n-1}{k}^2}{(n-k)i^2} \le \frac{4 e^2 (n-k) \binom{n-1}{k}^2}{\eps^2 k \binom{n-k}{k}^2} \le \frac{4e^2(n-k)}{\eps^2 k} \left( 1 + \frac{k}{n-2k} \right)^{2k} \le \frac{4e^2 (n-k) e^{k/(2C)}}{\eps^2 k} .\]
Since $k \le \frac{n}{6C}$, we may also bound $\left( n \binom{n-1}{k} \right)^{ - \zeta (n-2k)/(Cn)} \le \left( \frac{k}{n} \right)^{k/(2C)}$, and so the bases of the exponential summands in~\eqref{eqn:medYbound} are at most
\[ \frac{4e^2(n-k)}{\eps^2 k} \left( \frac{e k}{n} \right)^{k/(2C)} \le \frac{4e^3}{\eps^2} \left( \frac{e}{6C} \right)^{k/(2C) - 1} \le \frac{4e^3}{\eps^2 C^{k/(2C)-1}} \left( \frac{e}{6} \right)^2 < \frac{e^3}{\eps^2 C^{k/(3C)}} < \frac12, \]
as we may assume $k \ge 6C$ and $\eps \ge 10 C^{-k/(6C)}$.  This then implies $\sum_{i = t_0 + 1}^{t_1} \Prb(Y_i > 0) = o(1)$.

For families that are far from a star, we can re-estimate the upper bound in~\eqref{eqn:largeYbound} to show
\[ \sum_{i > t_1} \Prb(Y_i > 0) \le \left( \frac{ne}{k} \left(n \binom{n-1}{k} \right)^{ - \frac{\zeta (n-2k)}{(20C)^2 n} } \right)^{\binom{n-1}{k-1}} \le \left( \frac{ne}{k} \left( \frac{k}{n} \right)^{ \frac{k}{2(20C)^2}} \right)^{\binom{n-1}{k-1}} \le \left( \frac{ek}{n} \right)^{\binom{n-1}{k-1}} = o(1), \]
assuming $k \ge (40C)^2$.

To complete the proof of the sharp threshold, we must demonstrate that we are unlikely to obtain independent families that are very close to stars, with $1 \le i \le t_0$.  In this range, we repeat the analysis of Bollob\'as, Narayanan and Raigorodskii in~\cite{bnr14}, and instead consider \emph{maximal} independent families in $K_p(n,k)$.

For $j \ge i \ge 0$, let $Z_{i,j}$ denote the number of maximal independent families $\cF$ such that there is some star $\cS$ with $\card{\cS \setminus \cF} = i$ and $\card{\cF \setminus \cS} = j$.  Observe that if $\cF$ is a family counted by the random variable $Y_i$, then for any maximal independent family $\cF'$ containing $\cF$, the family $\cF'$ must be counted by $Z_{i',j}$ for some $i' \le i$ and $j \ge i$.  Furthermore, if $i' = 0$, this family $\cF'$ contains superstars.  This shows
\[ \bigcup_{1 \le i \le t_0} \left\{ Y_i > 0 \right\} \subseteq \left\{ X > 0 \right\} \cup \bigcup_{\substack{1 \le i \le t_0 \\ j \ge i}} \left\{ Z_{i,j} > 0 \right\}, \]
and so $\sum_{i=1}^{t_0} \Prb(Y_i > 0) \le \Prb(X > 0) + \sum_{i=1}^{t_0} \sum_{j \ge i} \Prb(Z_{i.j} > 0)$.  We already have $\Prb( X > 0 ) = o(1)$, and hence it suffices to show $\sum_{i = 1}^{t_0} \sum_{j \ge i} \Exp [Z_{i,j} ] = o(1)$.

Let $\cF$ be a maximal independent family counted by $Z_{i,j}$.  Let $\cS$ be the corresponding star, $\cA = \cS \setminus \cF$, and $\cB = \cF \setminus \cS$.  Thus we have $\card{\cA} = i$ and $\card{\cB} = j$.  By virtue of $\cF$ being independent, all of the edges between $\cB$ and $\cS \setminus \cA$ must be missing in $K_p(n,k)$.  As $\cF$ is maximal, each $A \in \cA$ must have an edge to some $B \in \cB$, for otherwise $\cF \cup \{ A \}$ would be a larger independent family.  In particular, this implies $\cB$ is a subset of the union of the neighbourhoods in $K(n,k)$ of $A \in \cA$.

There are thus $n$ choices for the star $\cS$, $\binom{\binom{n-1}{k-1}}{i}$ choices for $\cA$, and at most $\binom{i \binom{n-k}{k}}{j}$ choices for $\cB$.  Each $A \in \cA$ must retain at least one of its edges to $\cB$, which occurs with probability at most $jp$, independently for each of the $i$ sets.  Furthermore, as every $B \in \cB$ has $\binom{n-k-1}{k-1}$ neighbours in $\cS$, there are at least $j \left( \binom{n-k-1}{k-1} - i\right)$ edges between $\cB$ and $\cS \setminus \cA$ that must be missing.  This gives
\[ \Exp [ Z_{i,j} ] \le n \binom{ \binom{n-1}{k-1}}{i} \binom{i \binom{n-k}{k}}{j} (jp)^i \left(1 - p \right)^{j \left(\binom{n-k-1}{k-1} - i\right)} = z_{i,j}. \]

We first observe that, for $i \le t_0 = \frac{\eps}{2} \binom{n-k-1}{k-1}$ and $j \ge i$,
\begin{align*}
\frac{z_{i,j+1}}{z_{i,j}} = \frac{\binom{i \binom{n-k}{k}}{j+1} (j+1)^i}{\binom{i \binom{n-k}{k}}{j} j^i } \left( 1- p \right)^{\binom{n-k-1}{k-1} - i} &\le \frac{\left( i \binom{n-k}{k} - j \right) e}{j+1}\left( 1- p \right)^{\left(1 - \frac{\eps}{2} \right) \binom{n-k-1}{k-1}} \\
&\le e \binom{n-k}{k} \left( n \binom{n-1}{k} \right)^{- (1 + \frac{\eps}{4})} = o(1),
\end{align*}
since $p = \zeta p_c = \frac{ (1 + \eps) \log \left( n \binom{n-1}{k} \right) }{\binom{n-k-1}{k-1}}$.  Hence $\sum_{j \ge i} z_{i,j} \le 2 z_{i,i}$.  Thus
\begin{align*}
\sum_{i=1}^{t_0} \sum_{j \ge i} \Exp[Z_{i,j}] \le 2 \sum_{i=1}^{t_0} z_{i,i} &= 2n \sum_{i=1}^{t_0} \binom{\binom{n-1}{k-1}}{i} \binom{i \binom{n-k}{k}}{i} (ip)^i \left( 1 - p \right)^{i \left(\binom{n-k-1}{k-1} - i \right)} \\
	&\le 2n \sum_{i=1}^{t_0} \left( e^2 \binom{n-1}{k-1} \binom{n-k}{k} p \left(1 - p\right)^{\left( 1 - \frac{\eps}{2} \right) \binom{n-k-1}{k-1}} \right)^i \\
	&\le 2n \sum_{t=1}^{t_0} \left( e^2 \binom{n-1}{k-1} \binom{n-k}{k} \frac{ (1 + \eps) \log \left( n \binom{n-1}{k} \right) }{\binom{n-k-1}{k-1}} \left( n \binom{n-1}{k} \right)^{- (1 + \frac{\eps}{4} )} \right)^i \\
	&= 2n \sum_{i=1}^{t_0} \left( \frac{( 1 + \eps ) e^2 \log \left( n \binom{n-1}{k} \right)}{n \left( n \binom{n-1}{k} \right)^{\frac{\eps}{4} }} \right)^i = o(1),
\end{align*}
completing the proof for the sharp threshold.
\end{proof}

\section{Concluding remarks} \label{sec:conc}

In this paper, we built on the work of Filmus~\cite{fil14} to develop a removal lemma for large set families with few disjoint pairs.  We then used this to determine the threshold for random Kneser subgraphs having the Erd\H{o}s--Ko--Rado property, thus answering a question of Bollob\'as, Narayanan and Raigorodskii~\cite{bnr14}.

Rather than the probabilistic problem considered above, one might instead ask the corresponding extremal question: how sparse can a spanning subgraph $G$ of $K(n,k)$ be if $\alpha(G) = \binom{n-1}{k-1}$?  A lower bound can be obtained by requiring the stars to be maximal independent sets.  For every set $F \in \binom{[n]}{k}$, and every element $x \notin F$, there must be an edge between $F$ and the star $\cS_x$ centred at $x$, for otherwise $\cS_x \cup \{F\}$ would be an independent set of size $\binom{n-1}{k-1} + 1$.  As each edge $\{F, F'\}$ covers $k$ stars, it follows that $F$ must have degree at least $\frac{n-k}{k}$, and hence $G \subseteq K(n,k)$ must have at least $\frac{n-k}{2k} \binom{n}{k}$ edges.

Perhaps surprisingly, this simple lower bound can be tight.  If $k$ divides $n$, then Baranyai's Theorem~\cite{bar75} gives a partition of $\binom{[n]}{k}$ into perfect matchings.  In the Kneser graph, this corresponds to a partition of the vertices into cliques of size $\frac{n}{k}$.  Let $G$ be the subgraph consisting only of these cliques.  Any independent set in $G$ can contain at most one vertex from each clique, and hence $\alpha(G) \le \frac{k}{n} \binom{n}{k} = \binom{n-1}{k-1}$.  Furthermore, $G$ is $\frac{n-k}{k}$-regular, matching the lower bound given previously.

Theorem~\ref{thm:sparseEKR} shows that for this bound on the independence number to hold in random graphs, they must be denser by a factor of at least $\log \left( n \binom{n-1}{k} \right)$.  However, these random graphs have the additional property that the only maximum independent sets are the stars, which is not the case in the construction given above.  One might be interested in the extremal problem with this stricter requirement, or in the case when $k$ does not divide $n$.

Returning to the random setting, Devlin and Kahn~\cite{dk15} have recently established threshold results when $k \sim \frac{n}{2}$.  It remains to exhibit a sharp threshold around $p_c$ for $k > \frac{n}{6C}$.  We believe that, perhaps for smaller $k$, a more precise hitting time result may hold.  Consider the random process where one removes edges from the Kneser graph $K(n,k)$ one at a time, selecting at each step an edge uniformly at random from those that remain.  Is it true that, with high probability, $\alpha(G) > \binom{n-1}{k-1}$ precisely when a superstar is born?  The fact that the lower bound from the sharp threshold comes from these superstars suggests this might be the case.

More generally, given how central intersecting families are to extremal set theory, we believe the removal lemma should find many other applications.  In a forthcoming paper with Balogh, Liu and Sharifzadeh, we obtain some supersaturation results using the removal lemma with $\l \ge 1$, extending the results of~\cite{dgs14}.  We hope that the lemma might prove useful for other research directions as well.

\paragraph{Acknowledgements} We would like to thank Jozsef Balogh, Hong Liu and Maryam Sharifzadeh for suggesting the generalisation of the removal lemma to larger families, with $\l \ge 2$.  We also extend our thanks to the anonymous referees for their many helpful suggestions.

\end{document}